\documentclass{amsart} 
\usepackage{latexsym,amssymb,graphicx,amscd,amsmath}

\newtheorem{thm}{Theorem}[section]
\newtheorem{lem}[thm]{Lemma}
\newtheorem{prop}[thm]{Proposition}
\newtheorem{cor}[thm]{Corollary} 
\newtheorem{de}[thm]{Definition}

\newtheorem{que}[thm]{Question}

\newcommand{\BZ}{{\mathbb{Z}}}

\newcommand{\bbG}{{\mathbb{G}}}
\newcommand{\BO}{{\mathcal{O}}}
\newcommand{\BD}{{\mathbb{D}}}

\newcommand{\BB}{{\mathcal{B}}}

\newcommand{\Si}{{\Sigma}}

\newcommand{\bb}{{\mathfrak{b}}}
\newcommand{\bg}{{\mathfrak{g}}}

\DeclareMathOperator{\genus}{genus}

\begin{document}

\title[Heegaard genus, and quantum invariants]{Heegaard genus, cut number, weak $p$-congruence, and quantum invariants}

\author{ Patrick M. Gilmer}
\address{Department of Mathematics\\
Louisiana State University\\
Baton Rouge, LA 70803\\
USA}
\email{gilmer@math.lsu.edu}
\thanks{ partially supported by NSF-DMS-0604580,  NSF-DMS-0905736}
\urladdr{www.math.lsu.edu/\textasciitilde gilmer/}


\begin{abstract}  We use quantum invariants to define a 3-manifold invariant $j_p$ which lies in the non-negative integers. We relate $j_p$ to the Heegard genus, and the cut number. We show that $j_p$ is an invariant of weak $p$-congruence.
\end{abstract}

\maketitle

\section{Introduction}

Let $M$ be an oriented compact 3-manifold without boundary.
We use quantum invariants to give a lower bound on the Heegaard genus $g(M)$ of  $M$.
Let $p$ be an  prime greater than three and $q$ a primitive $pth$ root of unity. 
Let $\BO$ denote $\BZ[q]$ if $p \equiv 3 \pmod{4}$ and $\BZ[q,i]$ if $p \equiv 1 \pmod{4}$.  
Let $\BO^+$ denote $\BZ[q]$ in either case. We will use a version of the Witten-Reshitkhin-Turaev invariant $\mathcal{I}_p(M,L) \in \BO$ of a 3-manifold containing a framed link $L$.  Consider the $\BZ[q]$-ideal invariant:

\begin{de}\label{Jdef} \[J_p(M)= \left( \text{$\BO$-ideal generated by $\{ \mathcal{I}_p(M,L)| L \subset M\}$}\right) \cap \BO^+ \subset  \BO^+ .\]\end{de}

Let $h=1-q$.

\begin{thm}\label{upper}  $J_p(M) = h^{j_p(M)}$ for some non-negative integer $j_p(M).$
\end{thm}

This invariant, $j_p(M)$, is quite computable. See Theorem \ref{compute}, and the remark following it.
It follows from work of H. Murakami \cite{M}, that if $M$ is a $\BZ_p$-homology sphere, then $\mathcal{I}_p(M, \emptyset)$ is not divisible by $h$.  Thus $j_p(M)=0$, when $M$ is a $\BZ_p$-homology sphere.  However $j_p(M)$ is more interesting when $H_1(M,\BZ_p)$ is larger. Cochran and Melvin \cite{CM} studied the divisibility of $\mathcal{I}_p(M, \emptyset)$  by powers of $h$ with some lower bounds on the exponent.   At this point, we do not know whether their methods may be used to study $j_p(M)$.

\begin{prop} \label{add} $j_p(M\# M')= j_p(M)+j_p( M')$.
\end{prop}

\begin{de} The cut number $c(M)$ of  $M$ is the maximal number of disjoint compact surfaces without boundary  $F_i$ that one can place in $M$ with $M \setminus \cup F_i$ connected.
\end{de}

The $p$-cut number, $c_p(M)$, is defined similarly except the surfaces are allowed to have certain  $1$-dimensional singularities.
See the discussion below Proposition \ref{c} The co-rank of a group $G$ is the maximal rank for a free group which $G$ can map onto.
The cut number of $M$ is the co-rank of $\pi_1(M)$  \cite{J}.  Let $d= (p-1)/2.$
The following theorem is a synthesis of a number of results which will be proved separately.

\begin{thm}  \[ 0 \le c(M) \le c_p(M)\le \frac{j_p(M)}{d-1}   \le g(M)\]
\end{thm}

The ideal invariant $J_p$ first arose in the study of an equivalence relation among oriented compact 3-manifold without boundary, called weak $p$-congruence. This and two related but finer equivalence relations are studied in \cite{L,G2,G3}.

A $n/\ell$- surgery along a knot $K$ in a 3-manifold is the result of removing a tubular neighborhood of $K$ and regluing so that  a curve representing $ n  (\text{ the meridian}) +\ell (\text{ a longitude})$ bounds a disk in new solid torus.
The isotopy class of a longitude is not canonical but  any two longitudes differ by a multiple of the meridian.

\begin{de}   
If $\ell \equiv 0 \pmod{p}$, $n/\ell$-surgery is called  weak type-$p$ surgery
\end{de}

\begin{de} Two 3-manifolds are said to be  weakly $p$-congruence if one can pass from one to the other by a sequence of weak type-$p$ surgeries.
\end{de}

\begin{thm}\label{congI} If $M$ and $M'$ are weakly $p$-congruent, then $j_p(M)=j_p(M')$.  
\end{thm}

\begin{que}\label{ques} Is $j_p(M)$ always divisible by $d-1$?  
\end{que}

We thank  Gregor Masbaum, Neal Stoltzfus, and the referee for useful suggestions and comments.

\section{quantum invariants  }

Let $A= - q^{(p+1)/2}\in \BO^+ $ and $\kappa = \pm i^{p+1}A^{-3}\in \BO$. Then $A$ is a primitive $2p$th root of unity, and $\kappa$ is a primitive $4p$th root of unity.

Let $G$ be a $p$-admissibly colored banded trivalent graph in a 3-manifold $M$ \cite{BHMV2}. For instance, $G$ could be a framed link colored one. 
Let $\mathcal{L}$ be a framed link in $S^3$ that is a surgery description of $M$.  We may pick the surgery description so that $G$ lies in the complement of $\mathcal{L}$.
Let $\omega$ be the linear combination over $\BO$ of $p$-colored cores of a solid torus specified
in \cite{ BHMV2}.  Let $s(L)$ denote signature of the linking matrix of $L$. Let $\mathcal{L}(\omega)$ be the linear combination over $\BO$ that we obtain if we replace each component of $\mathcal{L}$ by $\omega.$  

\[ I_p((M,G))= \kappa^{- s(L)} [L(\omega) \cup G] \in \begin{cases}
\BO^+ &\text{$\beta_1(M)$ is even or $p\equiv -1 \pmod{4}$} \\
\kappa  \BO^+  \subset \BO &\text{$\beta_1(M)$ is odd and $p\equiv 1 \pmod{4}$}
\end{cases}\]

Here $[\  ]$ denotes the Kauffman bracket of the expansion \cite{KL, MV}.  That $ I_p((M,G))$ is an algebraic integer is due to H. Murakami \cite{M}, and Masbaum-Roberts.

We say two elements of $\BO$ agree up to phase if one is  a power of $\kappa$ times the other. One should think about Definition \ref{Jdef} as follows: $J_p(M)$ is the $\BO^+$-ideal generated by the $I_p$-invariant of all links in $M$ after adjusting them by multiplying by a power of $\kappa$, if necessary, so that they lie in $\BO^+$. 
Also it is true that $J_p(M)$ is the $\BO^+$-ideal generated by the $I_p$-invariant of all $p$-admissibly colored banded trivalent graphs $G$ in $M$  adjusted, as above, to lie in $\BO^+$.
It suffices in the definition to let ${G}$ vary over some set of generators for $K(M, \BO)$, the Kauffman skein module of $M$ over  $\BO$.

\begin{proof}[Proof of Proposition \ref{add}] Given a link $\hat L$ in $M \#M'$, we wish to show that 
\linebreak $\mathcal{I}_p(M \#M',\hat L)$ can be written a linear combination over $\BO$ of some products  of invariants of pairs of links in $M$ and $M'$. This is proved by induction on $m$, half the minimal number (after isotopy) of tranverse intersections of $\hat L$  with $S$,  the sphere where the connected sum takes place. If $m$ is zero, let $L$ be the part of $\hat L$ in $M$, and $L'$ be the part of $\hat L$ in $M'$. Then $\mathcal{I}_p(M \#M',\hat L)= \mathcal{I}_p(M, L)\ \mathcal{I}_p(M', L')$.  If $0 \le m < d$,  let $\check L$ denote $\hat L$ modified by inserting in the $m$ strands crossing $S$ a  Jones-Wenzl idempotent.  As $V_p(S, \text{point colored $m$})=0$,
$\mathcal{I}_p(M \#M',\check L)=0.$ By the recursive definition of the Jones-Wenzl idempotent, we have $\mathcal{I}_p(M \#M',\hat L)- \mathcal{I}_p(M \#M',\check L)$ is given by a linear combination of the invariants of links which intersect $S$ in fewer than $m$ points.  If $m\ge  d$, we can play the same game using just $p-1$ of the strands. Again we have $\mathcal{I}_p(M \#M',\check L)=0$, as the $p-1$st idempotent is ``zero as a map of outsides'', in the sense of Lickorish \cite{Li}. Thus in either case, we can reduce to a fewer number of strands.
 \end{proof}

\begin{proof}[Proof of Theorem \ref{congI}] Suppose that $M$ is obtained  from $N$ by a weak type-$p$ surgery along $\gamma$, and $\gamma$ is in the complement of some  $p$-admissibly colored banded trivalent graph
 $G$ in $M$. Continue to denote by $G$ the resulting  graph in $M$ after the surgery. According the \cite[Theorem 3.8]{G2},
for some integer $m$, 
\[ I_p((M,G))= \kappa^m I_p(
(N, \gamma \text{ with some color} \cup  G)
) \]
The result follows easily from this.
\end{proof}

Note that the phase factor in the above definition
and the related phase anomalie in TQFT are unimportant for the issues that we discuss in this paper. 
So, in our discussion of TQFT, we will ignore the extra structure needed to resolve the phase anomalie.
We  say two elements of $x,y $ of a $\BO$-module agree up to units if one  is a unit of $\BO$ times the other. In this case, we write $x \sim y$.

Related to the invariant $I_p$, we have a  TQFT \cite{BHMV2} which assigns a scalar $<(M,G)>_p \in \BO[\frac 1 h]$.
One has:
 \[I_p(M,G) = \frac {<(M,G)>_p}{<S^3 \text{ with no colored link} >} \sim h^{d-1} <(M,G)>_p\in \BO\]
  
The TQFT associates to a surface $\Si$ a free $\BO[1/h]$-module $V_p(\Si)$. This module comes with a Hermitian form
$< , >$. If  $\Si$ is the boundary of a handlebody $H$, then $V_p(\Si)$ is a quotient of the Kauffman bracket skein module of $H$ over
$\BO[\frac 1 h].$ Let $\bbG(H)$ be the image of the Kauffman bracket skein module of $H$ over $\BO^+$ under this map.

By a spine $S$ for $H$, we mean a banded (in sense of \cite{BHMV2})  trivalent graph embedded in $H$ that is a deformation retract of $H$. In fact, $H$ can be identified with  $P \times I$, where $P$ is a $2$-sphere with $\genus(\Si)+1$ open discs removed, and $S$ is embedded in 
$P \times \{1/2\}$, and $S$ is a deformation retract of $P \times \{1/2\}$, and the banding is given by $P \times \{1/2\}$. We call this a flat spine for $P \times I$. Then $S$ is dual to a triangulation of a  quotient of the planar surface $P$  where the boundary components have been crushed to points which serve as the vertices of the triangulation. According to  \cite[corollary]{H}, and references therein, one may pass from one such triangulation to any other by a sequence of  moves which exchange a pair triangles which meet along an edge by another pair.  This corresponds to the statement that one can pass between any two flat spines for $H=P \times I$ by a sequence of moves which replace a  H-shaped subgraph with the H-shaped subgraph rotated $90$ degrees. 
 
By a coloring of a trivalent graph, we will mean an assignment to each edge of a color from the set of ``colors'' $\{0,1,\cdots p-2\}$, which is $p$-admissible in the sense of \cite{BHMV2}.
 If $S$ is a spine for $H$, we consider the set of  colorings $\mathfrak{c}(S)=\{\mathfrak{c}_i \}$ of $S$.  These colorings represent elements of $\bbG(H)$ which we also denote by the same symbol.  The even colorings are the colorings where every edge of $S$ is given an even color. Let 
$\mathfrak{e}(S)=\{\mathfrak{e}_i \}$ denote the even colorings of $S$.  If $S$ is a lollipop tree  \cite{GM1}, we may also speak of  the small colorings $\mathfrak{g}(S)=\{\mathfrak{g}_i \}$ of $S$. These are the coloring which assign a small color (ie. colors less than or equal to $d-1$)
to the loop edges of the lollipop spine.   The non-loop edges of a lollipop spine for a surface (without colored points) are always assigned even colors.

 \begin{prop}\label{bbG} If $S$ is a lollipop spine for  $H$, then   $\mathfrak{g}(S)$, and $\mathfrak{e}(S)$ are bases for  $\bbG(H)$.  If  $S$  is a spine for $H$,  
 $\mathfrak{e}(S)$  is a basis for  $\bbG(H)$. 
 \end{prop}
 
 \begin{proof}
Let  $S$ be  a lollipop spine.  Using fusion and vanishing results in recoupling theory, one sees that  $\bbG(H)$ is generated by $\frak{c}(S)$. 
Up to units,  $\frak{g}(S)$ and $\frak{e}(S)$ are  the same subsets of $\bbG(H)$, see proof of \cite[Lemma (8.2)]{GMW}. Thus 
 each of these subsets  spans  $\bbG(H)$. As  $\frak{e}(S)$ is also a basis for  $V_p(\Si)$, the first claim holds.
 
 Let $S$ be  a (not necessarily) lollipop spine for $H$. View $S$ as a flat spine in  $P \times I$. Let $S'$ be a flat lollipop spine in $P \times I$.  
 As the matrix of $6j$ symbols which express  even colorings of an H-shaped graph as linear combinations of  
even colorings of that  H-shaped graph rotated $90$ degrees is invertible over $\BO$, and $S$ and $S'$ are related by such moves, the second claim holds. 
\end{proof}

 Order  $\mathfrak{g}(S)$, and $\mathfrak{e}(S)$  so that the zero coloring is first. Write $ \mathfrak{e}_1=  \emptyset$.
 Given $f$ in $\Gamma_\Si$, the mapping class group of $\Si$, the TQFT assigns a map $\rho_p(f): V_p(\Si) \rightarrow V_p(\Si)$.   
  
 \begin{thm}\label{compute} Suppose $S$ is a spine for $H$ and $\Si= \partial(H)$. If $M= H \cup_f  -H$ is a Heegard splitting of genus $g$ , then $J_p(M)$ is spanned by $h^{(d-1)g}$  times the entries in the first column of the matrix for  $\rho_p(f)$ with respect to $\frak{e}(S)$.\end{thm}
 
 \begin{proof} Let  $\mathcal{M}(f,\mathfrak{e}(S),p)$ denote this matrix whose entries  lie in $\BO^+[1/h].$   Any link $L$ in $M$ can be isotoped into $-H$ where it can be written (as `a map of outsides'), as an $\BO^+$-linear combination of elements from $\mathfrak{e}(S)$. Let $\partial H= \Si$.  Recall  that $\mathfrak{e}_i$ is orthogonal with respect to the Hermitian pairing  $\langle \ ,\ \rangle_{\Si}$ on  $V_p(\Si)$. Moveover 
 $\langle  \mathfrak{e_i} ,\mathfrak{e_i} \rangle \sim h^{(d-1)(g-1)}$, by \cite[Theorem 4.11]{BHMV2}. Also we have that
 
\begin{align} \notag I_p(M,\mathfrak{e}_i)&=  h^{(d-1)}<(M, \mathfrak{e}_i)> _p= h^{(d-1)}\langle \rho_p(f)(\mathfrak{e}_1), \mathfrak{e}_i \rangle \rangle_{\Si} \\
 \notag &= 
h^{(d-1)} \langle \sum \mathcal{M}(f,\mathfrak{e}(S),p)_{j,1}(\mathfrak{e}_j), \mathfrak{e}_i\rangle_{\Si} =
 h^{(d-1)}\mathcal{M}(f,\mathfrak{e}(S),p)_{,i,1}  \langle \mathfrak{e}_i, \mathfrak{e}_i  \rangle _{\Si}
 \\
 \notag &\sim h^{(d-1)g} \mathcal{M}(f,\mathfrak{e}(S),p)_{i,1}
 \end{align}
 As each basis element is an $\BO^+$-linear combination of links, we are done.
 \end{proof}

If $f \in \Gamma_\Si$ is written as a product of elements from a certain set of generating Dehn twists, then 
 A'Campo's package \cite{A'C} calculates $\mathcal{M}(f,\mathfrak{e}(S),p)$.  Thus it may be used to calculate $j_p(M)$ within Pari \cite{P}.

From now on, we hold fixed a choice of handlebody $H$ and a lollipop spine $S$ in $H$, and we sometimes omit mention of them from our notation and hypotheses. Thus for instance, we may just saying `coloring' instead of `coloring of $S$, a spine for $H.$'

\section{upper bounds on $j_p(M)$}

The integral version of this TQFT \cite{G1,GMW,GM1} assigns to a  surface $\Si$
a $\BO^+$-lattice  $\mathcal{S}^+_p(\Sigma) \subset V_p(\Sigma)$. 
$\mathcal{S}^+_p(\Sigma)$ is free with a  basis \cite{GM1}  $\BB= \{\bb(a,b,c) \}$  indexed by small    colorings
 $(a,b,c)$.  This is called the lollipop basis. 
It is obtained from the basis for $\bbG(H)$ given by small colorings of $S$ by a unimodular (for us,  unimodular means that the change of basis matrix  has a unit of $\BO^+$ as its determinant) triangular basis change followed by a rescaling by powers 
 of $h$.  In this paper we need only consider the case when $\Si$ has no colored points. In this case, the trunk half color denoted $e$ in \cite{GM1,GM2} is taken to be zero.

Besides the Hermitian pairing $< , >$ on $V_p(\Si)$, we consider a bilinear symmetric form, called the
Hopf pairing in \cite{GM2}:
 
\[((\quad, \quad)): V_p(\Sigma) \times V_p(\Sigma)\rightarrow \BO_p[1/h] \] 
which restricts to
\[((\quad, \quad)): \mathcal{S}^+_p(\Sigma) \times \mathcal{S}^+_p(\Sigma)\rightarrow \BO^+.\] 

If $x$, $y$ are represented by colored graphs $X$ and $Y$ in $H$, $((x,y ))$ is the bracket evaluation of the result  of leaving $X$ in $H$ but putting $Y$ in the complementary handlebody $H'$. See \cite{GM2}. 
 \begin{figure}[h]
\includegraphics[height=1.2in]{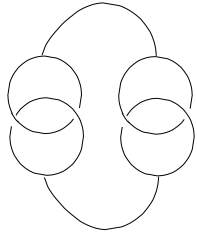}
\caption{In genus 2, H and $H'$ are neighborhoods of these graphs }
\end{figure} 
 
In \cite{GM2}, we defined a new basis for $\{ \tilde\bb(a,b,c)|\text{ $(a,b,c)$ is a small coloring of $G$} \}$   for $\mathcal{S}^+_p(\Sigma)$ , called the orthogonal lollipop basis.  It is obtained from the lollipop basis  by a unimodular  triangular basis change. This new basis is orthogonal with respect to the $((\quad, \quad))$.

 We use $\lfloor x \rfloor$ to denote the floor of $x$ i.e. the greatest integer less than or equal to  $x$. Similarly we use $\lceil x \rceil$ to denote the ceiling of $x$ i.e. the smallest integer greater than or equal to  $x$. Recall the eigenvalue for encircling a strand colored $i$,   $\lambda_i= -q^{i+1}-q^{-i-1}$. 

For a small  coloring $(a,b,c)$, define 
  \[\tilde\bg(a,b,c) = \bg(a,0,c) \prod_{j=1}^{g(M)} 
{\prod_{i=a_j}^{a_j+b_j-1}(z_j-\lambda_i)}~.  \]
  
  \begin{thm} $ \tilde \bg= \{ \tilde \bg(a,b,c)|\text{ $(a,b,c)$ is a small coloring} \}$ is a basis of $\bbG(H)$. 
 One has that  \[ \tilde \bg(a,b,c) 
  \sim h^{\lfloor \frac 1 2 (\sum_j a_j)\rfloor + \sum_j b_j } \cdot \tilde \bb(a,b,c)  .\]
  Moreover 
 \[\left(\left(\tilde \bb(a,b,c),\tilde \bg(a,b,c)\right) \right)\sim h^{ \lceil \frac  1 2 (\sum_j a_j)\rceil + \sum_j b_j } . \] 
 \end{thm}
 
 \begin{proof}  Since $\tilde \bg(a,b,c)$ is obtained from $ \bg(a,b,c)$  by  a unimodular triangular change of basis, the first claim holds.
 Claim 2 follows from \cite[Equation 6]{GM2}.   Then \cite[Thm (3.1),Cor(3.2)]{GM2} implies third claim.
 \end{proof}

Given  an element $x \in  \mathcal{S}^+_p(\Sigma)$, and a basis $\frak u= \{u_i\}$ of $\mathcal{S}^+_p(\Sigma)$ define $\phi_{\frak u}(x)$ to be the $\BO^+$-ideal generated by the coeficients of $x$ when written uniquely as a linear combination of  $\frak u.$

\begin{lem} If $\frak u= \{ u_i\}$ and $\frak u'  = \{ u'_i\}$ are two bases of   $\mathcal{S}^+_p(\Sigma)$ and  $x \in  \mathcal{S}^+_p(\Sigma)$ , then $\phi_{\frak u}(x)= \phi_{\frak u'}(x)$
\end{lem}

\begin{proof} One has that $u_i= \sum_j w_{i,j}  u'_j$.   If $x= \sum_i a_i   u_i$ , $x= \sum_j \sum_i w_{i,j} a_i  u_j'$. Then
$\phi_{\frak u}(x)$ is generated by $a_i $, and $\phi_{\frak u'}(x)$ is generated by $\sum_i w_{i,j} a_i $. Thus 
$\phi_{\frak u'}(x) \subset \phi_{\frak u}(x).$ By symmetry, we have $\phi_{\frak u}(x) \subset \phi_{\frak u'}(x).$
\end{proof}

We define  $\phi(x) =\phi_{\frak u}(x),$ where $\frak u$ is any basis of $\mathcal{S}^+_p(\Sigma).$
We have that  $\emptyset= \tilde \bb(0,0,0) \in \mathcal{S}_p(\Sigma)$.   
So $\phi(\emptyset)=(1).$ Moreover if $F$ is any automorphism of $\mathcal{S}^+_p(\Sigma)$, $\phi(F(\emptyset))=(1).$
 
We prove two assertions made in the introduction together.

\begin{thm}\label{genus} $J_p(M)= (h)^{j_p}$, for some integer $j_p(M)\le (d-1) g(M)$. 
\end{thm}

\begin{proof}
Suppose $M$ has a Heegaard splitting of genus $g$. 
Then $M$ can be obtained from the standard Heegaard genus $g$ splitting of $S^3$  by modifying the gluing map by
$f \in \Gamma(\Si)$. 
Write $\rho_p(f)(\emptyset)= \sum_{a,b,c} x_{a,b,c}  \tilde \bb(a,b,c)$.  Then the $\BO^+$-deal generated by $ \{ x_{a,b,c} \}_{a,b,c} $ is the trivial ideal $(1)$.
We have that  $J_p(M)$ is the ideal generated by $\{  I_p(M,  \tilde \bg(\alpha,\beta,\gamma) )\}_{\alpha,\beta,\gamma} $, where $(\alpha,\beta,\gamma)$ also varies through the small   colorings of our Lollipop spine for $H$. 

\begin{align}\notag& I_p(M,  \tilde \bg(\alpha,\beta,\gamma) )=(( \rho_p(f)(\tilde \bb(0,0,0)),   \tilde \bg(\alpha,\beta,\gamma) ))\\ \notag &= ( \sum_{a,b,c} x_{a,b,c}\   \tilde \bb(a,b,c) ,   \tilde \bg(\alpha,\beta,\gamma)  ))=  x_{\alpha,\beta,\gamma}\   \left(\left(\tilde \bb(\alpha,\beta,\gamma),\tilde \bg(\alpha,\beta,\gamma)\right) \right)\\ \notag & \sim
x_{\alpha,\beta,\gamma} h^{ \lceil \frac 1 2 (\sum_i \alpha_i)\rceil + \sum_j \beta_j}.
\end{align}

As $(d-1)g \ge \lceil \frac 1 2 (\sum_i \alpha_i)\rceil + \sum_j  \beta_j $ for all small colorings $(\alpha,\beta,\gamma )$, $J_p(M)$ includes the ideal generated by 
 $\{x_{a,b,c} h^{(d-1)g}\}$, but this is  $(1)(h^{(d-1)g})= (h^{(d-1)g}).$
This means that $J_p(M)| (h)^{(d-1)g}$.
 \end{proof}
 
\section{lower bounds on $j_p(M)$}

\begin{prop}\label{c} $j_p(M)\ge (d-1)c(M).$
\end{prop}

\begin{proof}By  \cite[Thm 15.1]{GM1} $I_p(M,G) \in h^{(d-1)c(M)} \BO.$
\end{proof}

This lower bound for $j_p(M)$ can sometimes be improved as follows.  
A $p$-surface is a generalized  surface which allows local singularities where $p$-sheets come together along simple closed curves called the 1-strata. See \cite{G2}.
A good $p$-surface is a $p$-surface where the neigborhood of the $1$-strata consists consists of the union of mapping cones of maps of a circle to a circle with degrees multiples of $p$. 
 The $p$-cut number of $M$  is the maximal number of disjoint good $p$-surfaces that you can embed in $M$ .

\begin{prop} \label{lower}$j_p(M)\ge {(d-1)c_p(M)} $.\end{prop}

\begin{proof} By  \cite[Theorem(4.2)]{G2}, $I_p(M,G) \in h^{(d-1)c_p(M)} \BO.$
\end{proof}

\section{Some Calculations and Consequences}
\begin{prop}\label{e} The following are equivalent:
\begin{enumerate}
\item  $M$ is a $\BZ_p$-homology sphere.
\item $j_p(M)=0$.
\item$j_p(M)< d-1.$
\item $c_p(M)= 0.$

\end{enumerate}
\end{prop}
\begin{proof}
(1) $\implies$ (2) follows from a theorem of H. Murakami \cite{M}.
(3) $\implies$ (4) by Proposition \ref{lower}.
(4) $\implies$ (1) by \cite[Prop 9]{GQ}.
\end{proof}

\begin{prop}
$j_p(\#^k S^1 \times S^2) =k(d-1)$.
\end{prop}
\begin{proof} $c(\#^k S^1 \times S^2)\ge k$ and $I_p( \#^k S^1 \times S^2)$ is a unit times $h^{k(d-1)}.$
\end{proof}

\begin{prop}
$j_p(L(pn,q)) =d-1$.
\end{prop}
\begin{proof} $L(pn,q)$ is weakly $p$-congruent to $S^1 \times S^2$. \end{proof}

The Breiskorn manifold $\Sigma(2,3,6)$ is a cohomology $\#^2 S^1 \times S^2$

\begin{prop}
$j_p(\Sigma(2,3,6)) = d-1$.
\end{prop}
\begin{proof} By Proposition \ref{e}, $j_p(\Sigma(2,3,6)) \ge d-1$. Also $I_p( \Sigma(2,3,6))$ is not divisible by $h^{d}$ \cite{G3,CM}.
\end{proof}

\begin{cor} $\Sigma(2,3,6)$ and $ \#^2 S^1 \times S^2$ are not weakly $p$-congruent.
\end{cor}
\begin{proof} \[ j_p(\Sigma(2,3,6)) = d-1 \neq 2(d-1)= j_p(\#^2 S^1 \times S^2)\]
\end{proof}

The above proof is just a rephrasing of a proof in \cite{G2}. This last result also follows from work of
Dabkowski-Przytycki using the Burnside group \cite{DP}. However the same results hold  for $\BD(h')$ of \cite[Theorem 3.15]{G3}. The Burnside method has not been tried on this manifold.

\begin{prop}
$j_p(S^1 \times S^1 \times S^1) = d-1$.
\end{prop}

\begin{proof} $c(S^1 \times S^1 \times S^1)\ge 1$ and $I_p( S^1 \times S^1 \times S^1) \sim \dim(V_p(S^1 \times S^1) h^{d-1}= d h^{d-1}.$
\end{proof}

\begin{prop}
$S^1 \times S^1 \times S^1$  is not weakly $p$-congruent to the connected sum of three $S^1 \times S^2$'s
\end{prop}

Two different proofs of the generalization of the above proposition, where $p$ is not required to be a prime (greater than 3) but is instead required to be an integer greater than 2, are given in \cite{G3}. In  \cite{G3},  I asked whether  $S^1 \times S^1 \times S^1$  is weakly $2$-congruent to the connected sum of three $S^1 \times S^2$'s.
Selman Akbulut has pointed out that they are weakly $2$-congruent and that this follows from  \cite[Theorem 3]{N}. See also \cite{AK}.

\end{document}